\documentclass[12pt]{amsart}
\usepackage{amscd}
\usepackage{graphics,epic} 
\def\NZQ{\Bbb}               
\def\NN{{\NZQ N}}

\def\ZZ{{\NZQ Z}}
\def\RR{{\NZQ R}}

\def\frk{\frak}               

\def\mm{{\frk m}}

\def\Phi{{\frk n}}
\def\Phi{{\frk N}}
\def\opn#1#2{\def#1{\operatorname{#2}}} 
\opn\chara{char}
\opn\length{\ell}
\opn\pd{pd}
\opn\rk{rk}
\opn\projdim{proj\,dim}
\opn\injdim{inj\,dim}
\opn\rank{rank}
\opn\depth{depth}
\opn\and{and}
\opn\grade{grade}
\opn\height{height}
\opn\embdim{emb\,dim}
\opn\codim{codim}

\opn\Tr{Tr}
\opn\bigrank{big\,rank}
\opn\superheight{superheight}\opn\lcm{lcm}
\opn\trdeg{tr\,deg}%
\opn\reg{reg}
\opn\lreg{lreg}
\opn\ini{in}
\opn\infpt{infpt}
\opn\Min{Min}
\opn\Ass{Ass}
\opn\div{div}
\opn\Div{Div}
\opn\cl{cl}
\opn\Cl{Cl}
\opn\Spec{Spec}
\opn\Supp{Supp}
\opn\supp{supp}
\opn\Sing{Sing}
\opn\Ass{Ass}
\opn\Ann{Ann}
\opn\Rad{Rad}
\opn\Soc{Soc}
\opn\Ker{Ker}
\opn\Coker{Coker}
\opn\Am{Am}
\opn\Hom{Hom}
\opn\Tor{Tor}
\opn\Ext{Ext}
\opn\End{End}
\opn\Aut{Aut}
\opn\id{id}

\opn\nat{nat}
\opn\pff{pf}
\opn\Pf{Pf}
\opn\GL{GL}
\opn\SL{SL}
\opn\mod{mod}
\opn\ord{ord}
\opn\cl{cl}
\opn\conv{conv}
\opn\ext{ext}
\opn\aff{aff}
\opn\con{conv}
\opn\relint{relint}
\opn\st{st}
\opn\lk{lk}
\opn\cn{cn}
\opn\core{core}
\opn\vol{vol}
\opn\link{link}
\opn\star{star}
\opn\gr{gr}

\def\pot#1#2{#1[\kern-0.28ex[#2]\kern-0.28ex]}

\opn\dirlim{\underrightarrow{\lim}}
\opn\inivlim{\underleftarrow{\lim}}
\let\iso=\cong

\let\Dirsum=\bigoplus

\newcommand{\ps}{\hspace{0.25ex}}
\let\to=\rightarrow

\def\Implies{\ifmmode\Longrightarrow \else
     \unskip${}\Longrightarrow{}$\ignorespaces\fi}
\def\implies{\ifmmode\Rightarrow \else
     \unskip${}\Rightarrow{}$\ignorespaces\fi}
\def\iff{\ifmmode\Longleftrightarrow \else
     \unskip${}\Longleftrightarrow{}$\ignorespaces\fi}

\let\:=\colon
\newtheorem{Theorem}{Theorem}[section]
\newtheorem{Lemma}[Theorem]{Lemma}
\newtheorem{Corollary}[Theorem]{Corollary}
\newtheorem{Proposition}[Theorem]{Proposition}

\let\epsilon\varepsilon
\let\phi=\varphi
\let\kappa=\varkappa
\def\qed{\ifhmode\textqed\fi
   \ifmmode\ifinner\quad\qedsymbol\else\dispqed\fi\fi}
\def\textqed{\unskip\nobreak\penalty50
    \hskip2em\hbox{}\nobreak\hfil\qedsymbol
    \parfillskip=0pt \finalhyphendemerits=0}
\def\dispqed{\rlap{\qquad\qedsymbol}}

\opn\dis{dis}
\def\pnt{{\raise0.5mm\hbox{\large\bf.}}}

\begin{document}

\title{Linear balls and the multiplicity conjecture}
\author{Takayuki Hibi and Pooja Singla }
\address{Takayuki Hibi, Department of Pure and Applied Mathematics,
Graduate School of Information Science and Technology,
Osaka University, Toyonaka, Osaka 560-0043, Japan.}
\email{hibi@math.sci.osaka-u.ac.jp}
\address{Pooja Singla, FB Mathematik, Universit\"at Duisburg-Essen, Campus Essen, 45117 Essen, Germany.}
\email{pooja.singla@uni-due.de}

\date{}
\maketitle
\begin{abstract}
A linear ball is a simplicial complex
whose geometric realization is homeomorphic to a ball
and whose Stanley--Reisner ring has a linear resolution.
It turns out that the Stanley--Reisner ring of the sphere
which is the boundary complex of a linear ball
satisfies the multiplicity conjecture.
A class of shellable spheres arising naturally from commutative algebra
whose Stanley--Reisner rings satisfy the multiplicity conjecture
will be presented.
\end{abstract}

\section*{Introduction}
The multiplicity conjecture due to Herzog, Huneke and Srinivasan is
one of the most attractive conjectures lying between combinatorics and commutative algebra.
First, we recall what the multiplicity conjecture says.

Let $R=\sum_{i = 0}^\infty R_i$ be a homogeneous
Cohen--Macaulay algebra over a field $R_0 = K$
of dimension $d$ with embedded dimension $n = \dim_K R_1$
and write $R = S / I$, where $S = K[x_1, \ldots, x_n]$ is the polynomial ring
in $n$ variables over $K$ and $I$ is a graded ideal of $S$.
Let
$H(R,i)=\dim_K R_i$, $i = 0, 1, 2, \ldots$, denote the Hilbert function of $R$
and $F(R,\lambda)=\sum_{i = 0}^\infty H(R,i)\lambda^i$ the
Hilbert series of $R$.
It is known that $F(R,\lambda)$ is a rational function of $\lambda$ of the form
\[
F(R,\lambda) =
\frac{h_0+h_1\lambda+\cdots+h_{\ell}\lambda^{\ell}}{(1-\lambda)^d},
\]
with each $h_i > 0$.
The multiplicity $e(R)$ of $R$ is
\[
e(R)= h_0 + h_1 + \cdots + h_\ell.
\]
Now, we consider the graded minimal free resolution
\[
0\longrightarrow F_p \longrightarrow \cdots \longrightarrow
F_1\longrightarrow S \longrightarrow R \longrightarrow 0
\]
of $R$ over $S$, where
$F_i=\bigoplus S(-j)^{\beta_{i,j}}$ with $\beta_{i,j}\geq 0$.
Let
\[
m_i=\min\{j: \beta_{i,j} \neq 0 \},
\, \, \, \, \, \, \, \, \, \,
M_i =\max\{j: \beta_{i,j} \neq 0 \}.
\]
The multiplicity conjecture due to Herzog, Huneke and Srinivasan
says that
\[
\frac{\prod_{i=1}^p m_i}{p!} \leq e(R) \leq  \frac{\prod_{i=1}^p M_i}{p !}.
\]

A nice survey of the multiplicity conjecture and the record of past results in 
different cases of the conjecture can be found in \cite{HZ}. For more recent results one may look into
\cite{welker}, \cite{tim},  \cite{novik}. 

In the present article we discuss the problem of finding
a natural class of spheres whose
Stanley--Reisner rings satisfy the multiplicity conjecture.

Let $\Delta$ be a simplicial complex on the vertex set
$[n] = \{ 1, \ldots, n \}$ of dimension $d - 1$ and
$K[\Delta] = S / I_\Delta$,
where $S = K[x_1, \ldots, x_n]$, its Stanley--Reisner ring.
Suppose that $\Delta$ is a ball, i.e., the geometric realization
$|\Delta|$ is a ball.
Let $\partial\Delta$ denote the boundary complex of $\Delta$
and suppose that each vertex of $\Delta$ belongs to $\partial\Delta$.
Thus $\partial\Delta$ is a sphere, i.e., the geometric realization
$|\partial\Delta|$ is a sphere, of dimension $d - 2$ on $[n]$.
Each face of $\partial\Delta$ is called a boundary face of $\Delta$
and each face of $\Delta\setminus \partial\Delta$ is
called an inside face of $\Delta$.
Let $m - 1$ denote the smallest dimension of a nonface of $\Delta$
and suppose that $2\leq  m \leq[(d + 1) / 2]$.
It turns out (Theorem \ref{mul}) that
the sphere $\partial\Delta$ satisfies the multiplicity conjecture
with assuming the hypothesis that
\begin{enumerate}
\item[(A1)] $\Delta$ has a minimal inside face of dimension $d - m$
and has no minimal inside face of dimension less than $m-1$;
\item[(A2)] the $h$-vector of $\partial\Delta$ is unimodal.
\end{enumerate}
A linear ball is a ball whose Stanley--Reisner
ring has a linear resolution.
It is shown that
the sphere which is the boundary complex of a linear ball
satisfies (A1) and (A2).
In particular the Stanley--Reisner ring of the sphere
which is the boundary complex of a linear ball
satisfies the multiplicity conjecture
(Corollary \ref{linmul}).

A class of shellable spheres satisfying (A1) and (A2) arises
from determinantal ideals.
Let $X=(X_{ij})_{1 \leq i \leq m \atop 1 \leq j \leq n}$
be an $m \times n$ matrix of indeterminates,
where $m \leq n$.
Write $\tau$ for the lexicographic order of the polynomial ring
$K[X]=K[\{ X_{ij} \}_{1 \leq i \leq m \atop 1 \leq j \leq n}]$
induced by the ordering of the variables
\[
X_{11}\geq X_{12}\geq \cdots \geq X_{1n} \geq X_{21} \geq \cdots \geq X_{2n}
\geq \cdots \geq
X_{m1} \geq  \cdots \geq X_{mn}.
\]
Let $I_r$ denote the ideal of $K[X]$ generated by
all $(r+1) \times (r+1)$ minors of $X$, where $1 \leq r \leq m-1$.
In particular $I_{m-1}$ is the ideal of $K[X]$
generated by all maximal minors of $X$.
It is known that the initial ideal $I_r^*$ of $I_r$ with respect to $\tau$
is generated by squarefree monomials.  Let $\Delta_r$ denote the simplicial
complex whose Stanley--Reisner ideal coincides with $I_r^*$.
Theorem \ref{minor} says that, for each $1 \leq r \leq m-1$,
the simplicial complex
$\Delta_r$ is a shellable ball satisfying (A1) and (A2).
Moreover $\Delta_r$ is a linear ball if and only if $r = m-1$ (Corollary \ref{linear}).

One of the natural classes of shellable linear balls arises from
the polarization of a power of the graded maximal ideal.
Let $\mm = (x_1, \ldots, x_n)$ be the graded maximal ideal
of $S = K[x_1, \ldots, x_n]$.
Each power $\mm^t$ of $\mm$ has a linear resolution.
Let $\Delta$ be the simplicial complex whose Stanley--Reisner ideal
coincides with the polarization of $\mm^t$.
It is shown (Theorem \ref{polar}) that $\Delta$ is a shellable linear ball
for $t \geq 0$ and hence it satisfies the multiplicity conjecture.

\section{The Multiplicity Conjecture}
First, we recall fundamental material
on Stanley--Reisner ideals and rings of simplicial complexes.
We refer the reader to \cite{BH}, \cite{Hibi}, \cite{ST} for  further information.
Let $[n] = \{ 1, \ldots, n \}$ be the vertex set and
$\Delta$ a simplicial complex on $[n]$.  Thus $\Delta$ is
a collection of subsets of $[n]$ such that
\begin{enumerate}
\item[(i)] $\{ i \} \in \Delta$ for all $i \in [n]$, and
\item[(ii)] if $F \in \Delta$ and $F' \subset F$, then $F' \in
\Delta$.
\end{enumerate}
Each element $F \in \Delta$ is called a {\em face} of $\Delta$.
The dimension of a face $F$ is $|F| - 1$. Let $d = \max\{ |F| : F
\in \Delta \}$ and define the dimension of $\Delta$ to be
$\dim \Delta =  d - 1$.
A {\em nonface} of $\Delta$ is a
subset $F$ of $[n]$ with $F \not\in \Delta$.

Let $f_i = f_i(\Delta)$ denote the number of faces of $\Delta$ of
dimension $i$.  Thus in particular $f_0 = n$. The sequence
$f(\Delta) = (f_0, f_1, \ldots, f_{d-1})$ is called the
{\em $f$-vector} of $\Delta$. Letting $f_{-1} = 1$, we define
the {\em $h$-vector} $h(\Delta) = (h_0, h_1, \ldots, h_d)$ of
$\Delta$ by the formula
\[
\sum_{i=0}^{d} f_{i-1} (t - 1)^{d-i} = \sum_{i=0}^{d} h_i t^{d-i}.
\]
Let $S = K[x_1, \ldots, x_n]$ denote the polynomial ring
in $n$ variables over a field $K$
with each $\deg x_i = 1$.
For each subset $F \subset [n]$, we set
\[
x_F = \prod_{i \in F} x_i.
\]
The {\em Stanley--Reisner ideal} of $\Delta$ is the ideal
$I_\Delta$ of $S$ which is generated by those squarefree monomials
$x_F$ with $F \not\in \Delta$. In other words,
\[
I_\Delta = ( x_F : F \not\in \Delta).
\]
The quotient ring $K[\Delta] = S / I_\Delta$ is called
the {\em Stanley--Reisner ring} of $\Delta$.
It follows that the Hilbert series of $K[\Delta]$ is 
\[
F(K[\Delta],\lambda) =
(h_0 + h_1 \lambda + \cdots + h_d \lambda^d)/(1 - \lambda)^d,
\]
where $(h_0, h_1, \ldots, h_d)$ is the $h$-vector of $\Delta$.
Thus in particular the multiplicity of $K[\Delta]$ is
$\sum_{i=0}^{d} h_i$ $( = f_{d-1} )$.

We say that $\Delta$ is Cohen--Macaulay (resp. Gorenstein)
over $K$ if $K[\Delta]$ is Cohen--Macaulay (resp. Gorenstein).
If the geometric realization $|\Delta|$
of $\Delta$ is homeomorphic to a ball, then $\Delta$ is Cohen--Macaulay
over an arbitrary field.  
If the geometric realization $|\Delta|$
of $\Delta$ is homeomorphic to a sphere, then $\Delta$ is Gorenstein
over an arbitrary field.  

Now, let $\Delta$ be a  simplicial complex on $[n]$ of dimension
$d - 1$ whose geometric realization $|\Delta|$ is homeomorphic to a manifold.
The {\it boundary complex } $\partial\Delta$
of $\Delta$ consists of those faces $F$ of $\Delta$
with the property that
there is a $(d-2)$-dimensional face $F'$ of $\Delta$ with $F \subset F'$
such that $F'$
is contained in exactly one $(d-1)$-dimensional face of $\Delta$.
Each face of $\partial\Delta$ is called a {\it boundary face}
and each face of
$\Delta\setminus \partial\Delta$ is called an {\it inside face} of $\Delta$.
In particular if $\Delta$ is a ball, i.e.,
$|\Delta|$
is homeomorphic to a ball, of dimension $d - 1$, then
$\partial\Delta$ is a sphere, i.e.,
$|\partial\Delta|$
is homeomorphic to a sphere, of dimension $d - 2$.
\begin{Theorem}[Hochster]\label{hoch} 
Let $\Delta$ be a Cohen--Macaulay complex 
over a field $K$  of dimension $d - 1$
whose geometric realization $|\Delta|$ is a manifold 
with a nonempty boundary complex $\partial\Delta$,
and let $\omega_{\Delta}$ be the canonical ideal of  
$K[\Delta]$.
Write $J$ for the ideal of $K[\Delta]$ generated
by those monomials $\overline{x_F}$ with $F \in \Delta \setminus \partial\Delta$.  
Then the following conditions are equivalent:
\begin{enumerate}
\item[(a)] $\omega_{\Delta} \iso J$ as a $\ZZ^n$-graded $K[\Delta]$-module;
\item[(b)] $\partial\Delta$ is a Gorenstein complex over $K$.
\end{enumerate}
If the equivalent conditions hold, then $K[\partial\Delta]\iso K[\Delta]/\omega_{\Delta}.$

\end{Theorem}



Let $\Delta$ be a simplicial complex on $[n]$ of dimension
$d - 1$ whose geometric realization $|\Delta|$ is a ball
and $\partial\Delta$ its boundary complex.  
Assume that every vertex of $\Delta$ belongs to $\partial\Delta$.
Thus
$\partial\Delta$ is a simplicial complex on $[n]$ of dimension
$d - 2$ whose geometric realization $|\partial\Delta|$ is a sphere. 
Since $\partial\Delta$ is Gorenstein, it follows that  
\begin{enumerate}
\item[(P1)] The $h$-vector
$h(\partial\Delta) = (h'_0, h'_1, \ldots, h'_{d-1})$
of $\partial\Delta$ is symmetric i.e. $h'_i=h'_{d-1-i}$ for all $i=0,\ldots,d-1$;
see \cite[Theorem 5.4.2, Theorem 5.6.2]{BH}.
\item[(P2)] The minimal free resolution of the Stanley--Reisner ring of
$\partial\Delta$ is symmetric (\cite[Corollary 21.16]{E}), i.e.
if
\[
0\longrightarrow F_{p} \longrightarrow \cdots \longrightarrow
F_1\longrightarrow F_0 \longrightarrow S/I_{\partial\Delta} \longrightarrow 0
\]
is the minimal free resolution of the ring $S/I_{\partial\Delta}$, where
$F_i=\Dirsum_{j}S(-j)^{\beta_{i,j}}$, $i=0,\ldots,p$, $p=n-(d-1)$ and $F_0=S$, then
we have 
$\beta_{i,j}=\beta_{p-i,n-j}$ for all $i=0,\ldots,p$.
In particular,
$M_i=n-m_{p-i}$ where $M_i=\max\{j \:\ \beta_{i,j} \neq 0\}$ and 
$m_i=\min\{j \:\ \beta_{i,j}\neq 0\}$.
\item[(P3)] The canonical ideal $\omega_{\Delta}$ of the Stanley--Reisner ring $K[\Delta]
= S / I_{\Delta}$ is generated
by the monomials $\overline{x_F},\ F \in \Delta \setminus \partial\Delta$ (see Theorem \ref{hoch}).
\end{enumerate}
In addition,  
\begin{enumerate}
\item [(F1)] 
Let 
\[
0\longrightarrow F'_{n-d} \longrightarrow \cdots \longrightarrow
F'_1\longrightarrow F'_0 \longrightarrow S/I_{\Delta} \longrightarrow 0
\]
be the minimal free resolution of  $S/I_{\Delta}$ 
with $F'_i=\Dirsum_{j}S(-j)^{\beta'_{i,j}}$. Then
the generators of the canonical module $\omega_{\Delta}$ of $K[\Delta]$
are of degrees $n-j$ with $\beta'_{n-d,j}\neq 0$ (see \cite[Corollary 3.3.9]{BH}).
\item [(F2)] One has $m_1 < m_2 < \cdots < m_{n-d+1}$.  
\end{enumerate}
Now, let $m - 1$ denote the smallest dimension of the nonfaces
of $\Delta$.  In other words, $m$ is the smallest degree of
monomials belonging to $G(I_{\Delta})$,
the minimal system of monomial generators of $I_{\Delta}$.  
We will assume that $2\leq m \leq [(d + 1) / 2]$.
Our goal is to show that
the Stanley--Reisner ring $K[\partial\Delta] = S/I_{\partial\Delta}$
satisfies the multiplicity conjecture under the following hypothesis (Theorem \ref{mul}):
\begin{enumerate}
\item[(A1)] $\Delta$ has a minimal (under inclusion) inside face of dimension $d-m$ and
has no minimal inside face of dimension less than $m-1$;
\item[(A2)] The $h$-vector of the
boundary complex $\partial\Delta$ is unimodal.
\end{enumerate}

(In general, we say that a finite sequence of real numbers
$a_0,\ldots,a_t$ is {\it unimodal}
if
\[
a_0\leq a_1\leq \cdots\leq a_j\geq a_{j+1}\geq \cdots \geq a_t
\]
for some
$0\leq j\leq t$.)

Now, we wish to understand the minimal and maximal shifts given by $m_i$ and $M_i$  respectively
 of the minimal free resolution
$$ \mathcal{F}_{\partial\Delta}: 0\longrightarrow F_{n-d+1}\longrightarrow\cdots\longrightarrow F_1
\longrightarrow S \longrightarrow S/I_{\partial\Delta} \longrightarrow 0 $$
of $S/I_{\partial\Delta}$ where $F_i=\bigoplus_{j}S(-j)^{\beta_{i,j}}$, to calculate the lower and upper bounds of the
multiplicity of $S/I_{\partial\Delta}$. First, we consider the minimal free resolution
$$\mathcal{F}_{\Delta}: 0\longrightarrow F'_{n-d}\longrightarrow\cdots\longrightarrow F'_1
\longrightarrow S \longrightarrow S/I_\Delta \longrightarrow 0$$
of
$S/I_{\Delta}$ where $F'_i=\bigoplus_{j}S(-j)^{\beta'_{i,j}}$. Let $m'_i$ and $M'_i$ denote the minimal and maximal
shifts of the minimal free resolution $\mathcal{F}_{\Delta}$.
Since $m$
is the minimum of the degree of generators of $I_{\Delta}$, one has $m'_1=m$.
By the assumption (A1) on $\Delta$,
there exists a minimal inside face of $\Delta$ of dimension $d-m$, hence by Theorem \ref{hoch}, it follows
 that
 the canonical ideal $\omega_{\Delta}$ of $\Delta$ has a generator of degree $d-m+1$. Therefore
 $\beta'_{n-d,n-(d-m+1)}\neq 0$, by (F1). As we have $m'_1=m$ and $m'_{n-d} \leq m+n-d-1$,
 we get $m'_i= m+i-1$ for $i=1,\ldots,n-d$, by (F2).

 We claim that the minimal shifts  in the minimal free resolution $\mathcal{F}_{\partial\Delta}$
 of $S/(I_{\partial{\Delta}})$  are given by $m_i=m+i-1$ for $i=1,\ldots, n-d$  and $m_{n-d+1}=n$.
 Indeed, by assumption
 (A1), we have that the canonical ideal ${\omega_\Delta}$ has no generator of degree less than $m$. Hence
the $S$-module
 $I_{\partial\Delta}/I_{\Delta}$  has no generator of degree less than $m$ (Theorem \ref{hoch}).
 From the following short
 exact sequence
 $$0\longrightarrow I_{\Delta} \longrightarrow I_{\partial\Delta}\longrightarrow
 I_{\partial\Delta}/I_{\Delta}\longrightarrow 0,$$
 we get the following long exact sequence
 \begin{eqnarray*}
 \cdots\longrightarrow\Tor_{i+1}(I_{\partial\Delta}/I_{\Delta},K)\\
 \longrightarrow\Tor_{i}(I_{\Delta},K)\longrightarrow \Tor_{i}(I_{\partial\Delta},K)
 \longrightarrow \Tor_{i}(I_{\partial\Delta}/I_{\Delta},K)\longrightarrow\cdots
 \end{eqnarray*}
 
Now, as $\Tor_i(I_{\Delta},K)_{i+t}=0$  and
 $\Tor_{i}(I_{\partial\Delta}/I_{\Delta},K)_{i+t}= 0$ for $t\leq m-1$ and $i=1,\ldots,n-d$,
 from the above long exact sequence 
 we get $\Tor_i(I_{\partial\Delta},K)_{i+t}= 0$ for $t\leq m-1$ and $i=1,\ldots,n-d$. Also as
 $\Tor_{i+1}(I_{\partial\Delta}/I_{\Delta},K)_{i+1+m-1}=0$ and $\Tor_{i}(I_{\Delta},K)_{i+m}\neq 0$,
 we get $\Tor_i(I_{\partial\Delta},K)_{i+m}\neq 0$, $i=1,\ldots n-d$. From here it follows that $m_i=m+i-1$ for
 $i=1,\ldots,n-d$. Since $S/I_{\partial\Delta}$
 is Gorenstein and $m_0=M_0=0$, we have $m_{n-d+1}=M_{n-d+1}=n$ by Property (P2).\pagebreak[3]

 Now, we need to determine the maximal shifts $M_i$ for $i=1,\ldots,n-d$
 in the minimal free resolution $\mathcal{F}_{\partial\Delta}$ of $S/I_{\partial\Delta}$.
 Again, as $S/I_{\partial\Delta}$ is Gorenstein, by Property (P2) we have
 $M_i=n-m_{n-d+1-i}=n-(m+n-d+1-i-1)=d-m+i$ for $i=1,\ldots,n-d$.

 Hence, we have now
 \begin{eqnarray*}
L&=&\prod_{i=1}^{n-d+1}\frac{m_i}{(n-d+1)!}\;=\;\frac{n\prod_{i=1}^{n-d} (m+i-1)}{(n-d+1)!}\;\; \and\\
U&=&\prod_{i=1}^{n-d+1}\frac{M_i}{(n-d+1)!}\;=\;\frac{n\prod_{i=1}^{n-d}(d-m+i)}{(n-d+1)!}.
\end{eqnarray*}
Next, our goal is to estimate the multiplicity $e(S/I_{\partial\Delta})$ of the ring $S/I_{\partial\Delta}$.
Let $h'_0,\ldots,h'_{d-1}$ denotes the $h$-vector of the ring $S/I_{\partial\Delta}$. As the ring
$S/I_{\partial\Delta}$ is Cohen-Macaulay, and $m$ is the minimum of the degree of the generators of
$I_{\partial\Delta}$, we have $h'_i=h'_{d-1-i}=\binom{n-d+1+i-1}{i}=\binom{n-d+i}{i}$ for $i=0,\ldots,m-1$.
From assumption (A2) and property (P1) we have that the $h$-vector is symmetric and unimodal, therefore
we conclude that $h'_i\geq \binom{n-d+m-1}{m-1}$ for $i=m,\ldots,d-(m+1)$.\pagebreak[3]

Hence $$e(S/I_{\partial\Delta})=\sum_{i=1}^{d-1}{h_i}\geq 2\sum_{i=0}^{m-1}\binom{n-d+i}{i}+(d-2m)
\binom{n-d+m-1}{m-1}.$$\pagebreak[3]

\begin{Theorem}\label{mul} Let $\Delta$ be a ball and 
$\partial\Delta$ be its  boundary complex. Suppose that the sphere
$\partial\Delta$ satisfies the assumptions $(A1)$ and $(A2)$. Then the Stanley-Reisner ring
$S/\partial\Delta$ satisfies the multiplicity conjecture i.e.
 $$L \leq e(S/I_{\partial\Delta}) \leq U.$$
 \end{Theorem}\pagebreak[3]
 For the proof of the theorem, we need to first define cyclic polytopes. Let $C(n,d-~1)$ denote the
 convex hull of any $n$ distinct points in $\RR^{d-1}$ on the curve 
 $\{(t,t^2,\ldots,t^{d-1})\in \RR^{d-1}, t\in \RR\}$.
 The polytope $C(n,d-1)$ is called the cyclic polytope of dimension $d-1$. It is known that $C(n,d-1)$
 is simplicial (i.e., every proper face is a simplex),
 and so the boundary  of $C(n,d-1)$ defines a simplicial complex which we denote by
 $\partial C(n,d-1)$ such that $|\partial C(n,d-1)|$ is a sphere of dimension $d-2$.
 Let $(h^*_0,h^*_1,\ldots,h^*_{d-1})$ denote the $h$-vector of $\partial C(n,d-1)$. Then
 $$h^*_i=h^*_{d-1-i}=\binom{n-d+i}{i} \;\;\mbox{for}\;\; i=1,\ldots,\lfloor \frac{d-1}{2} \rfloor,$$
 (see \cite[Section 3]{ST}). Let $e(\partial C(n,d-1))=\sum h^*_i$ denotes the multiplicity of the 
 Stanley-Reisner ring of the boundary complex
 $\partial C(n,d-1)$. Notice that we have $h'_i \leq h^*_i$, hence
 \begin{equation} \label{com}
 e(S/I_{\partial\Delta}) \leq e\big(\partial C(n,d-1)\big).
 \end{equation}
 In \cite{TH}, the  minimal free resolution of the $\partial C(n,d-1)$ is computed. We have the following
 \cite[Theorem 3.2]{TH}:
If $d-1 \geq 2$ is even, then the maximal shifts $M_i^*$ in the minimal free resolution of
 $\partial C(n,d-1)$ are given by
 \begin{equation}
 M_i^*=\frac{d-1}{2}+i \;\;\mbox{for}\;\;i=1,\ldots, n-d \;\; \and \;\;M_{n-d+1}^*=n
 \end{equation}
 and if $d-1 \geq 3$ is odd, then the maximal shifts $M_i^*$  are as follows:
 \begin{equation}
 M_i^*=\lfloor\frac{d-1}{2}\rfloor+i+1 \;\;\mbox{for}\;\;i=1,\ldots,n-d \;\; \and\;\; M_{n-d+1}^*=n.
 \end{equation}
Even though the following Lemma \ref{cyclic} follows from \cite[Theorem 1.2]{HS}, we want to give a direct 
computational proof. 

\begin{Lemma}  
\label{cyclic} We have
\begin{equation}\label{up}
e\big(\partial C(n,d-1)\big) \leq \frac{\prod_{i=1}^{n-d+1} M_i^*}{(n-d+1)!}.
\end{equation}
\end{Lemma}

\begin{proof} Let $U=\frac{\prod_{i=1}^{n-d+1} M_i^*}{(n-d+1)!}$. 
Let first $d-1\geq 2$ is even. Then
$$U=\frac{n(\frac{d}{2}+\frac{1}{2})(\frac{d}{2}+\frac{3}{2})\cdots(n-\frac{d}{2}-\frac{1}{2})}
{(n-d+1)!}.$$
We have the multiplicity
\begin{eqnarray*} 
&&e\big(\partial C(n,d-1)\big)=\sum_{i=0}^{d-1}h^*\\
&=&2\big[\binom{n-d+0}{0}+\cdots+\binom{(n-d)+d/2-3/2}{d/2-3/2}\big]+\binom{(n-d)+d/2-1/2}{d/2-1/2}\\
&=&2\binom{n-d/2-1/2}{d/2-3/2}+\binom{n-d/2-1/2}{d/2-3/2} \\
&=&\frac{2(n-d/2-1/2)\cdots(d/2-1/2)}{(n-d+1)!}+\frac{(n-d/2-1/2)\cdots(d/2+1/2)}{(n-d)!}\\
&=&\frac{(n-d/2-1/2)\cdots(d/2+1/2)}{(n-d+1)!}(d-1+n-d+1)\\
&=&U.
\end{eqnarray*}
Now let $d-1\geq 3$ be odd. Then $$U=\frac{n\big(\frac{d}{2}+1\big)\cdots\big(\frac{d}{2}+(n-d)\big)}{(n-d+1)!}.$$
And the multiplicity is given by
\begin{eqnarray*} 
e\big(\partial C(n,d-1)\big)&=&\sum_{i=0}^{d-1}h^*\\
&=&2\big[\binom{n-d+0}{0}+\binom{n-d+1}{1}+\cdots+\binom{n-d+d/2-1}{d/2-1}\big]\\
&=&2\binom{n-d/2}{d/2-1}\\
&=&2\frac{(n-d/2)\cdots(d/2+1)(d/2)}{(n-d+1)!}.\\
\end{eqnarray*}
We see that $e\big(\partial C(n,d-1)\big)\leq U$ if and only if $d \leq n$ which is true.
\end{proof}
\begin{proof} [Proof of Theorem \ref {mul}] 
Since $m \leq [(d+1)/2]$, we have $M_i^* \leq  M_i$ both when $d$ is odd and even.
Hence, by Equation (\ref{com}) and
Equation (\ref{up}), we get
\begin{equation}\label{fo}
e(S/I_{\partial\Delta}) \leq \frac{\prod_{i=1}^{n-d+1} M_i}{(n-d+1)!}.
\end{equation}
It remains to show that $ e(S/I_{\partial\Delta})\geq L$.
Since $$e(S/I_{\partial\Delta})\geq 2\sum_{i=0}^{m-1}\binom{n-d+i}{i}+(d-2m)
\binom{n-d+m-1}{m-1},$$ it is enough to show that
$$ 2\sum_{i=0}^{m-1}\binom{n-d+i}{i}+(d-2m)
\binom{n-d+m-1}{m-1}\geq \frac{n\prod_{i=1}^{n-d} (m+i-1)}{(n-d+1)!}$$
which is to prove
$$2\binom{n-d+m}{m-1}+(d-2m)\binom{n-d+m-1}{m-1}\geq \frac{n\prod_{i=1}^{n-d} (m+i-1)}{(n-d+1)!}.$$
We need to show
\begin{eqnarray*}
2(n-d+m)\cdots(m+1)(m)+(d-2m)(n-d+m-1)\cdots(m+1)(m)(n-d+1)\\
\geq n(m)(m+1)\cdots(m+n-d-1)
\end{eqnarray*}
which further amounts to prove that
$2(n-d+m)+(d-2m)(n-d+1)\geq n$.
Notice that it is enough to show that $2(n-d+m)+(d-2m)\geq n$ which is true as $n >d.$
 \end{proof}
\begin{Corollary}\label{linmul} Let $\Delta$ be a linear ball. Then the simplicial sphere $\partial\Delta$
satisfies the
multiplicity conjecture.
\end{Corollary}
\begin{proof} We only need to show that the assumptions (A1) and (A2) are satisfied in this case.
Since $S/I_{\Delta}$ has a linear resolution, the minimal and maximal shifts in the
minimal free resolution of $S/I_{\Delta}$ are given by $m_i'=M_i'=m+i-1$ for $i=1,\ldots,n-d$. Hence
$\Delta$ has inside faces only of dimension $n-(m+n-d-1)-1=d-m$, by fact (F1) and
Theorem \ref{hoch}. Also, there is no inside face of dimension less than $m-1$ since $d-m\geq m-1$.
Hence the assumption (A1) is satisfied.
We now show that the
 $h$ vector $(h'_0,\ldots,h'_{d-1})$  of
$S/I_{\partial\Delta}$ is unimodal.
 As the Stanley-Reisner ideal $I_{\Delta}$ has linear resolution
and  $S=K[\Delta]=S/I_{\Delta}$ is Cohen-Macaulay, we get that the $h$-vector
$(h_0,\ldots,h_{d})$ of $S/I_{\Delta}$
 is given by $h_i=\binom{n-d+(i-1)}{i}$
for $i=0,\ldots,m-1$ and $h_i=0$ for $i \geq m$. \pagebreak[3]

Now the $h$-vector of $S/I_{\partial\Delta}$ is equal to (see \cite[p. 137]{ST})\;:
$$(h_0-h_d,h_0+h_1-h_d-h_{d-1},\ldots,h_0+\cdots+h_{d-1}-h_d-\cdots-h_1).$$
Hence the $h$-vector of $S/I_{\partial\Delta}$
is given by
$$
h^{'}_i =
\begin{cases}
\binom{n-d+i}{i} & \text{for $i=0,\ldots,m-2$;}\\\\
\binom{n-d+m-1}{m-1} & \text{for $i=m-1,\ldots, d-m$;}\\\\
\binom{n-d+(d-1-i)}{d-1-i} & \text{for $i=d-m+1,\ldots,d-1$.}
\end{cases}
$$
Hence the assumption (A2) also holds.
\end{proof}
\section{Determinantal Ideals} 

In this section, we study simplicial complexes arising from determinantal ideals. It is known that these
simplicial complexes are shellable. We prove that the geometric realization of these
simplicial complexes are balls and these balls are linear only in the case of the ideal of maximal minors.
We show that the boundary complexes of these simplicial complexes satisfy the multiplicity conjecture.
\pagebreak[3]

Let $X=(X_{ij})$, $i=1,\ldots,m$, $j=1,\ldots,n$, $m \leq n$ be an $m \times n$ matrix of indeterminates.
We denote by $[a_1,\ldots,a_r\ps|\ps b_1,\ldots,b_r]$, the minor $\det(X_{a_{i}b_{j}})$ of $X$ where
$i,j=1,\ldots,r$. Further we define
$$[a_1,\ldots,a_r\ps|\ps b_1,\ldots,b_r]\leq [a_1^{'},\ldots,a_s^{'}\ps|\ps b_1{'},\ldots,b_s^{'}],$$
if $r \geq s $ and $a_i\leq a_i^{'}$, $b_i\leq b_i^{'}$ for $i=1,\ldots,s$. Let $\Delta(X)$
denote the poset of minors of $X$. For $\sigma=[a_1,\ldots,a_r\ps|\ps b_1,\ldots,b_r]\in\Delta(X)$, 
we denote by $I_{\sigma}$ the ideal generated
by all minors $\gamma \not\geq\sigma$. We call such ideals determinantal ideals.
Notice that for $\sigma=[1,\ldots,r\ps|\ps 1,\ldots,r]$, $r \leq m-1$,
the ideal $I_{\sigma}$ is the ideal generated by all $(r+1)\times (r+1)$ minors of $X$. For 
$\sigma=[1,\ldots,r\ps|\ps 1,\ldots,r]$, $r \leq m-1$, we denote the ideal $I_{\sigma}$ by $I_r$. Note that 
the ideal $I_{m-1}$ is generated by all maximal minors of $X$.\pagebreak[3] 

Let the symbol $\tau$ denote the lexicographic term order on the polynomial ring
$S=K[X]=K[X_{ij},\;i=1,\ldots,m,\;j=1,\ldots,n]$ induced by
the variable order
$$X_{11}\geq X_{12}\geq \cdots \geq X_{1m} \geq X_{21} \geq X_{22} \cdots \geq X_{2m} \geq 
X_{n1} \geq X_{n2} \geq  \cdots \geq X_{mn}.$$ Notice that under the monomial order $\tau$,
the initial monomial of any minor of $X$ is the product of the elements of its main diagonal.
Such a monomial order is called diagonal order. 
In \cite{HT}, it is shown that the generators of $I_{\sigma}$ form a Gr\"obner basis and hence 
$I_{\sigma}^{*}$ of $I_{\sigma}$ with respect to the monomial order $\tau$,
is generated by squarefree monomials. In other words,
$K[X]/I_{\sigma}^{*}$ may be viewed as a Stanley-Reisner ring of a certain 
simplicial complex $\Delta_{\sigma}$. For $\sigma=[1,\ldots,r\ps|\ps 1,\ldots,r]$, $r \leq m-1$, we denote
the simplicial complex $\Delta_{\sigma}$ by $\Delta_{r}$.
\pagebreak[3]

We show in
Theorem \ref{minor}
that for any $\sigma=[a_1,\ldots,a_r\ps|\ps b_1,\ldots,b_r]\in\Delta(X)$, the geometric 
realization $|\Delta_{\sigma}|$ of the
simplicial complex $\Delta_{\sigma}$ is a shellable ball. By Theorem~\ref{minor} and Corollary \ref{linear} 
together, it follows that 
the geometric realization $|\Delta_{m-1}|$ of $\Delta_{m-1}$
is in fact a shellable linear ball. \pagebreak[3]

According to \cite{HT}, the facets of simplicial complex $\Delta_\sigma$ can be described as follows: 
its vertex set
is the set of coordinate points $V=\{(i,j): 1 \leq i \leq m, 1 \leq j \leq n\}.$ We define a partial order on  
$V$
by setting 
$(i,j) \leq (i^{'},j^{'})$ if $i \geq i^{'}$ and $j \leq j^{'}$.
A maximal chain in  $V$ will be called a {\it path}. \pagebreak[3]

\begin{Theorem}\cite[Theorem 3.3]{HT} Let $ \sigma=[a_1,\ldots,a_r\ps|\ps b_1,\ldots,b_r]$, and let $P_i=(a_i,n)$
and $Q_i=(m,b_i)$ for $i=1, \cdots,r$. Then the facets of $ \Delta_{\sigma}$
are the non-intersecting paths from $P_i$ to $Q_i$, that is, subsets 
$C_1\cup C_2 \cup  \cdots \cup C_r$ of $V$ where each $C_i$ is a path with end 
points $P_i$ and $Q_i$ and where $C_i \cap C_j =  \emptyset$ for all $ i \neq j$.
\end{Theorem}\pagebreak[3]

We denote the set of facets of $\Delta_{\sigma}$ by $\mathcal{F}(\Delta_{\sigma})$.
The complex $\Delta_{\sigma}$ has a natural partial order on the set of facets which we recall from \cite[Theorem 4.9]{HT}: 
Let $F_1$
and $F_2$ be two facets of $\Delta_{\sigma}$. We write $F_1=\bigcup_{i=1}^{r}C_i$ and $F_2=\bigcup_{i=1}^{r}D_i$
as unions of non-intersecting paths with end points $P_i$ and $Q_i$.  We say that $F_2 \geq F_1$, if
$D_i$ is contained in the upper right side of $C_i$ for all $i=1,\ldots,r$, that is, if for
each $(x,y) \in D_i$ there is some $(u,v) \in C_i$  such that $u \leq x$  and $v \leq y$, where 
$i=1,\ldots,r$. 
This is a partial order
on the facets of $\Delta_{\sigma}$, and this partial order extended to any linear order gives us a shelling.
We fix a linear order and let $\Sigma$ denotes the  corresponding shelling.
From \cite[Corollary 5.18]{BV}, we have $\dim(S/I_{\sigma}^{*})=r(m+n+1)-\sum_{i=1}^{r}(a_i+b_i)$.\pagebreak[3]

Before stating the next theorem, we define the notion of a ${\it corner }$ of a path. Let $C$
be a path in $V$. A point $(i,j) \in C$ will be called a ${ \it corner}$ of $C$, if $(i-1,j)$
and $(i,j-1)$ belong to $C$.  Let $F$ be a facet of $\Delta_{\sigma}$, then we denote  by $\mathcal{C}(F)$,
the set of corners of the paths in $F$, and we define $c(F)=|\mathcal{C}{(F)}|$. \pagebreak[3]

For the proof of Theorem \ref{minor}, we need the following lemma from algebraic topology:
\begin{Lemma}\label{topology} Let $E_{1}$ be a simplicial complex  whose 
geometric realization $|E_1|$ is a ball of dimension $d$, and let $E_{2}$ be 
a simplex of dimension $d$.
Let the intersection $E_1 \cap E_2 =\langle G_1,\ldots,G_r \rangle \neq \emptyset$, where 
$G_1,\ldots,G_r$ are facets of the 
boundary complexes $\partial E_i$
of $E_i$, $i=1,2$ and $\langle G_1,\ldots,G_r \rangle$ is a proper subset of $\partial E_2$.
 Then the geometric realization $|E_1 \cup E_2|$  of $E_1 \cup E_2$ is again a ball.
\end{Lemma}\pagebreak[3]

The following lemma follows 
from the proof of \cite[Theorem 2.4]{BAH}.

\begin{Lemma}\label{late} Let $\Delta_{\sigma}=\langle F_1,\ldots,F_t \rangle$ be the simplicial complex 
with Stanley-Reisner ideal $I_{\sigma}$ where  $F_1,\ldots,F_t$ is the shelling order $\Sigma$. 
Let 
$\Delta_{i}=\langle F_1,\ldots, F_i \rangle$ and let $G=F_k\setminus\{v\}$ for some $v \in F_k$, $k \leq i$.
Then $G\subset F_{\ell}$ for some $\ell< k$
if and only if $v \in \mathcal{C}(F_k)$. If the equivalent conditions hold then $F_{\ell}$
is uniquely determined.
\end{Lemma}\pagebreak[3]
\begin{Theorem}\label{minor} For any $\sigma=[a_1,\ldots,a_r\ps|\ps b_1,\ldots,b_r]\in\Delta(X)$, the geometric 
realization $|\Delta_{\sigma}|$ of the
simplicial complex $\Delta_{\sigma}$ is a shellable ball of dimension
$r(m+n+1)-\sum_{i=1}^{r}(a_i+~b_i)-1$. 
\end{Theorem}
\begin{proof} The fact that the
dimension of the simplicial complex $\Delta_{\sigma}$
is $r(m+n+1)-\sum_{i=1}^{r}(a_i+~b_i)-1$ follows from 
\cite[Corollary 5.18]{BV}.
Let $\Delta_{\sigma}=\langle F_1,\ldots,F_t \rangle$ where $F_1,\ldots,F_t$ is the shelling order
$\Sigma$.
Let $\Delta_{i}=\langle F_1, \ldots, F_i \rangle$. 
We prove that $|\Delta_{i}|$ is a ball by induction on $i$. Assume that $|\Delta_{i-1}|$
is a ball, we will show that $|\Delta_{i}|$ is a ball. We have
$\Delta_{i}=\Delta_{i-1} \cup \langle F_{i} \rangle$,
let $\Delta_{i-1} \cap \langle F_{i} \rangle = \langle G_1, \ldots, G_r \rangle$. Notice that
$G_j$
are codimension one faces of $\Delta_{i-1}$ as $\Delta_{\sigma}$ is shellable. 
By Lemma \ref{topology}, we notice that $|\Delta_{i}|$ is a ball
(assuming that $|\Delta_{i-1}|$ is a ball), if the following two conditions are satisfied:
\begin{enumerate}
\item Each $G_j$ is a subset of exactly one $F_k$ for $k \leq i-1$, which in turn implies
that $G_{j} \in \partial{\Delta_{i-1}}$,
\item $G_1, \ldots, G_r$ is a proper subset of the boundary complex $\partial \langle F_{i}\rangle$ of 
$\langle F_{i}\rangle$.
\end{enumerate}

The first condition follows from Lemma \ref{late}.
For the second condition, we define $G_{v}=F_{i} \setminus \{v\}$ where 
${v} \notin \mathcal{C}(F_{i})$ (Notice that such a $v$ exists as not all points in $F_{i}$
are corner points of $F_{i}$). Then again from Lemma \ref{late}, there exists no $F_{j}$, $j\leq i-1$ 
such that $G_v = F_{j} \cap F_{i}$. Hence 
$G_v \subset \partial \langle F_{i}\rangle$ and $G_v \neq G_j$ for $j=1,\ldots, r$.
\end{proof}

An ideal $I \subset S$ generated in degree $d$ is said to have a linear resolution if 
in the minimal free resolution of $I$, one has the maximal shifts $M_i=d+i$ for all $i$.
It is known that the ideal $I_{m-1}$ generated by the maximal minors of matrix $X$
has a linear resolution. In fact, the Eagon--Northcott complex  gives a minimal free 
resolution for $I_{m-1}$, see \cite[Theorem 2.16]{BV}.  We have the following :

\begin{Corollary}\label{linear} Let $\Delta_r$ be the simplicial complex with the Stanley-Reisner Ideal
$I_r^*$. Then $|\Delta_r|$ is a linear ball if and only if $r=m-1$.
\end{Corollary}
\begin{proof} First we show that $|\Delta_{m-1}|$ is a linear ball i.e. we show that
the Stanley Reisner ideal $I_{m-1}^*$ has a linear resolution. As stated before, we know that the ideal
$I_{m-1}$ has a linear resolution. Moreover, the ring $S/I_{m-1}$ is Cohen-Macaulay,
see \cite[Theorem 2.8]{BV}. Now as $\Delta_{m-1}$ is shellable, the ring 
$S/I_{m-1}^*$ is also Cohen-Macaulay. From here it follows, that the Stanley-Reisner ideal
$I_{m-1}^*$ also has a linear resolution. Indeed, note that $S/I_{m-1}$ 
and $S/I_{m-1}^*$ have the same Hilbert function. Let 
$\dim S/I_{m-1}=~\dim S/I_{m-1}^*=~d$. Let $y_1,\ldots,y_d$ and $y'_1,\ldots,y'_d$ be the
maximal regular sequences of linear forms in $S/I_{m-1}$ and in $S/I_{m-1}^*$,
respectively. Then $\overline{S}/\overline{I_{m-1}}$ is zero dimensional \big(here 
$\overline{\phantom{1}}$ denotes
modulo the sequence $(y_1,\ldots,y_d)$\big) and has a linear
resolution. This is only possible if $\overline{I_{m-1}}$ is a power of the maximal ideal  of 
$\overline{S}$.  Now the zero dimensional ring $\overline{S}/\overline{I_{m-1}^*}$
\big(here $\overline{\phantom{1}}$ denotes modulo the sequence $(y'_1,\ldots,y'_d)$\big)
has the same Hilbert function as $\overline{S}/\overline{I_{m-1}}$. This is only
possible if $\overline{I_{m-1}^*}$ is the same power of the maximal ideal as
$\overline{I_{m-1}}$. In particular, $\overline{I_{m-1}^*}$ has linear resolution, and therefore
$I_{m-1}^*$ has a linear resolution.

Now we show that $I_r^*$ does not have a linear resolution for $r \neq m-1$. Notice that it is enough to show 
that $I_r$ does not have linear resolution for $r \neq m-1$, since
$\beta_{ij}(I_r^*)\geq \beta_{ij}(I_r)$. The $a$-invariant of the ring $S/I_{r}$ is equal to
$-nr$ i.e. the minimum of the degree of generators of the canonical module of $S/I_r$ is
given by $nr$, see \cite[Corollary 1.5]{BAH}. As the projective dimension of
$S/I_r$ is given by $(m-r)(n-r)$\cite[Corollary 5.18]{BV}, we have 
$M_{(m-r)(n-r)}(S/I_r)= nm-rn$ by (F1) in the first section. Hence $M_{(m-r)(n-r)-1}(I_r)-(m-r)(n-r)+1=
nm-rn-(m-r)(n-r)+1=r(m-r)+1$ and $M_0(I_r)=r+1$.
Hence for $r \neq m-1$, the ideal $I_r$ does not have a linear resolution.
 \end{proof}

The Stanley-Reisner ring $S_{\sigma}=K[\Delta_{\sigma}]$ being Cohen-Macaulay, admits a graded canonical module
$\omega_{\sigma}$. In \cite{BAH}, the $a-$invariant of $S_{\sigma}$  which is the negative of the 
least degree of canonical module $\omega_{\sigma}$ is computed. Next, we want to determine the degree of all the
generators of $\omega_{\sigma}$ for $\sigma=[1,\ldots,r\ps|\ps 1,\ldots,r]$, $r\leq m-1$. First we need 
the following lemma:

\begin{Lemma}\label{inout} Let $\Delta_{\sigma}=\langle F_1,\ldots,F_t \rangle$ be the simplicial complex 
with Stanley-Reisner ideal $I_{\sigma}$ and $F_1,\ldots,F_t$ be the shelling order $\Sigma$. 
Let $\Delta_{i}=\langle F_1,\ldots, F_i \rangle$. Then the boundary complex of $\Delta_i$ is given by
$$\partial(\Delta_i)=\big\{G\in \Delta_i:\; F_k\setminus G \not\subset \mathcal{C}(F_k)\;
\mbox{for\;all\;}\; k\leq i\; \mbox{with}\; G\subset F_k\big\}.$$ 
\end{Lemma}\pagebreak[3]

\begin{proof}  It is enough to show that the set of facets of $\partial(\Delta_i)$
is given by 
$$\mathcal{F}(\partial(\Delta_i))=\big\{G\in \Delta_i:
F_k\setminus G=\{v\}, v\notin\mathcal{C}(F_k)\;\mbox{for\;all} \;k\leq i\; \mbox{with}\;G\subset F_k\big\}.$$
Indeed, if we assume the above statement to be true, then 
the boundary complex is the set:
$$\big\{H\in \Delta_i:\; H\subset G \;\mbox{for}\;\mbox{some}\; G\in \mathcal{F}\big(\partial(\Delta_i)\big)\big\},$$
which is further equal to the set
$$\big\{H\in \Delta_i:\; H\subset G,\;F_k\setminus G=\{v\},\; 
v\notin\mathcal{C}(F_k)\;\mbox{for\;all} \;k\leq i\; \mbox{with}\;G\subset F_k \big\}.$$
The above set is equal to 
$$\big\{H\in \Delta_i:\; F_k\setminus H\not\subset\mathcal{C}(F_k)\; \mbox{for\;all}\; k\leq i\; \mbox{with}\;
H\subset F_k\big\},$$
as in the statement of the lemma.

Let $\mathcal{S}=\big\{G\in \Delta_i:
F_k\setminus G=\{v\}, v\notin\mathcal{C}(F_k)\;\mbox{for\;all} \;k\leq i\; \mbox{with}\;G\subset F_k\big\}$.
By Lemma~\ref{late}, we have
$\mathcal{S}\subset\mathcal{F}\big(\partial(\Delta_i)\big)$. 
Now let $G\notin \mathcal{S}$ be of  codimension one. It follows that $G$ 
is of the form $F_k\setminus\{v\}$ where $v\in \mathcal{C}(F_k)$ for some $k\leq i$.
Again by Lemma \ref{late}, there exists $\ell<k$ such that $G\subset F_{\ell}$.
Hence $G=F_{\ell}\cap F_{k}$, which implies $G\notin\partial(\Delta_i)$.
\end{proof}

In Theorem \ref{minor}, we have shown that the geometric realization $|\Delta_{\sigma}|$ 
of $\Delta_{\sigma}$ is a ball and therefore the geometric realization $|\partial_{\sigma}|$ of
$\partial_{\sigma}$ is a sphere. It is known that simplicial spheres are Gorenstein over any field, 
 see \cite[Corollary 5.6.5]{BH}. Hence we may apply Theorem \ref{hoch} to compute $\omega_{\sigma}$.
Before stating the next corollary, we define the notion of a {\it non-flippable} path. Let $D$
be a path from $a$ to $b$. Let $v\in D$ such that $\{v+(1,0),v+(0,1)\}\in D$
and neither $v+(1,0)$ nor $v+(0,1)$ is a corner point of $D$. Then $v$ 
can be flipped to get a path $D'=(D\setminus\{v\})\cup \{v+(1,1)\}$. We call such 
an interchange of the point $v$ to $v+(1,1)$ {\it a flip}. 
Notice that the new path $D'$ obtained after a flip from $D$ has the following property:
 $\mathcal{C}(D)\subset \mathcal{C}(D')$. We call a path $D$ to be a {\it flippable} path if 
 $D$ could be flipped to get a new path $D'$, otherwise we call $D$ to be a {\it non-flippable} path.
 Hence, a non-flippable 
path $D$ from $a$ to $b$ is a path which has the following property: 
for all $v\in D$ such that $\{v+(0,1),v+(1,0)\}\subset D$, one has 
 either $v+(0,1)$ or $v+(1,0)$ is a corner point of $D$. Equivalently, one may notice that a path 
$D$ from $a$ to $b$ is a  non-flippable path if for a path $D'$ from $a$ to $b$
with $\mathcal{C}(D') \supset \mathcal{C}(D)$, one has $D'=D$. 

 \begin{figure}[h]
\setlength{\unitlength}{6mm}
\begin{picture}(6,7)
\linethickness{0.15mm}
\multiput(-6,0)(1,0){6}%
{\line(0,1){6}}
\linethickness{0.15mm}
\multiput(-6,0)(0,1){7}%
{\line(1,0){5}}
\multiput(-4.8,2.1)(1,0){1}%
{{$v$}}
\multiput(-6,6.1)(1,0){1}%
{{$D$}}
\linethickness{0.4mm}
\multiput(-6,6)(0,1){1}%
{\line(0,-1){2}}
\linethickness{0.4mm}
\multiput(-6,4)(1,0){1}%
{\line(1,0){1}}
\linethickness{0.4mm}
\multiput(-5,4)(0,1){1}%
{\line(0,-1){2}}
\linethickness{0.4mm}
\multiput(-5,2)(1,0){1}%
{\line(1,0){2}}
\linethickness{0.4mm}
\multiput(-3,2)(1,0){1}%
{\line(0,-1){1}}
\linethickness{0.4mm}
\multiput(-3,1)(1,0){1}%
{\line(1,0){1}}
\linethickness{0.4mm}
\multiput(-2,1)(1,0){1}%
{\line(0,-1){1}}
\linethickness{0.4mm}
\multiput(-2,0)(1,0){1}%
{\line(1,0){1}}

\linethickness{0.15mm}
\multiput(6,0)(1,0){6}%
{\line(0,1){6}}
\linethickness{0.15mm}
\multiput(6,0)(0,1){7}%
{\line(1,0){5}}
\linethickness{0.4mm}
\multiput(6,6)(0,1){1}%
{\line(0,-1){2}}
\linethickness{0.4mm}
\multiput(6,4)(1,0){1}%
{\line(1,0){1}}
\linethickness{0.4mm}
\multiput(7,4)(0,1){1}%
{\line(0,-1){1}}
\linethickness{0.4mm}
\multiput(7,3)(0,1){1}%
{\line(1,0){1}}
\linethickness{0.4mm}
\multiput(8,3)(0,1){1}%
{\line(0,-1){1}}
\linethickness{0.4mm}
\multiput(8,2)(1,0){1}%
{\line(1,0){1}}
\linethickness{0.4mm}
\multiput(9,2)(1,0){1}%
{\line(0,-1){1}}
\linethickness{0.4mm}
\multiput(9,1)(1,0){1}%
{\line(1,0){1}}
\linethickness{0.4mm}
\multiput(10,1)(1,0){1}%
{\line(0,-1){1}}
\linethickness{0.4mm}
\multiput(10,0)(1,0){1}%
{\line(1,0){1}}
\multiput(8.2,3.1)(1,0){1}%
{{$v'$}}
\multiput(6,6.1)(1,0){1}%
{{$D'$}}

\end{picture}
\caption{\label{flnfl} A flippable path $D$ and a non-flippable
path $D'$ where $D'=(D\setminus\{v\})\cup\{v'\}$.}
\end{figure}
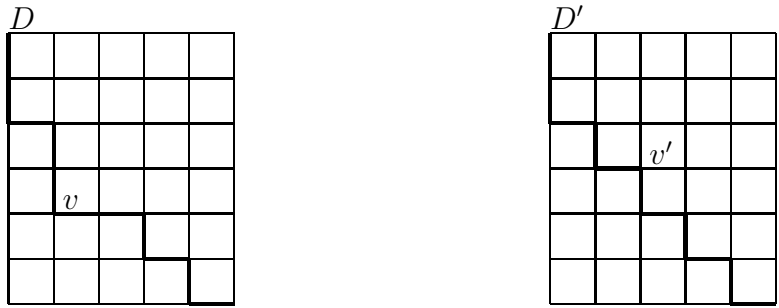

We call a facet $F=\bigcup_{i} C_i$ of the simplicial complex $\Delta_{\sigma}$ a 
{\it non-flippable facet}, if each $C_i$
is a non-flippable path, otherwise we call $F$ a {\it flippable facet}. Notice that
a facet $F$ of $\Delta_{\sigma}$ is non-flippable if for each facet $F'$ of $\Delta_{\sigma}$
with $\mathcal{C}(F') \supset \mathcal{C}(F)$, one has $F'=F$. 
We denote the set
of non-flippable facets of $\Delta_{\sigma}$ by $\mathcal{NF}(\Delta_{\sigma})$.  
Let $F,F'$ be two facets of $\Delta_{\sigma}$ with $\mathcal{C}(F)\subset \mathcal{C}(F')$.
Then $F'$ is obtained from $F$ by finite number of flips. One has:
\begin{Lemma} \label{fun} Let $F,F'$ be two facets of $\Delta_{\sigma}$, then the following two
conditions are equivalent:
\begin{enumerate} 
\item[(a)] $\mathcal{C}(F)\subset\mathcal{C}(F')$,
\item[(b)] $F'\setminus \mathcal{C}(F') \subset F \setminus \mathcal{C}(F)$.
\end{enumerate}
\end{Lemma}

For a given subset
$Z$ of $[m]\times[n]$ we denote by $X_Z$, the monomial $\prod_{(i,j)\in Z}X_{ij}$. 
We have :

\begin{Corollary} \label{mincan} Let $\omega_{\sigma}$ be the canonical ideal  of $K[\Delta_{\sigma}]$
and $\mathcal{M}$ denote the set $\{F \setminus~\mathcal{C}({F}):\; F\in \mathcal{NF}(\Delta_{\sigma})\}$. 
Then the minimal set of generators of 
$\omega_{\sigma}$ is given by
$G(\omega_{\sigma})=\{X_{G}:\; G \in \mathcal{M}\}$.
\end{Corollary}
\begin{proof} 
By Theorem \ref {minor} and Theorem \ref{hoch}, it is enough to show that $\mathcal{M}$ is the set of  
the minimal inside faces (under inclusion) of $\Delta_{\sigma}$.

By Lemma \ref{inout}, we know that the set of inside faces of the simplicial complex 
$\Delta_{\sigma}$ is given by
$\mathcal{S}=~\{F \setminus~Z:\;  F \in \mathcal{F}(\Delta_{\sigma}),\; Z \subset \mathcal{C}(F)\}$. 
Therefore each minimal inside face $G$
is of the form $F\setminus\mathcal{C}(F)$, $F \in \mathcal{F}(\Delta_{\sigma})$.

Let $F\in \mathcal{NF}(\Delta_{\sigma})$. Suppose $G=F\setminus \mathcal{C}(F)$
is a not a minimal inside face. Then there exists $G'\subset G$ such that $G'=F'\setminus \mathcal{C}(F')$
is a minimal  inside face. By Lemma \ref{fun}, it follows $\mathcal{C}(F') \supset \mathcal{C}(F)$, a 
contradiction.

Now, let $G=F\setminus \mathcal{C}(F)$ be a minimal inside face. Suppose 
$F\notin \mathcal{NF}(\Delta_{\sigma})$, then there exists  a facet $F'$ such that 
$\mathcal{C}(F')\supset \mathcal{C}(F)$. Again, by Lemma \ref{fun}, it follows then
$F'\setminus \mathcal{C}(F')\subset F\setminus \mathcal{C}(F)$, a contradiction.
\end{proof}

 In general, to give the explicit expressions of multi-degrees of the generators of canonical ideal
 $\omega_{\sigma}$ may  not be possible. But we would like to give all possible total degrees of 
 the generators of the canonical ideal  $\omega_{\sigma}$ for $\sigma=[1,\ldots,r\ps|\ps 1,\ldots,r]$, $r \leq m-1$. 
 In this case, $I_{\sigma}$ is the ideal generated by all $r+1\times r+1$ minors of $X$. For
 $\sigma=[1,\ldots,r\ps|\ps 1,\ldots,r]$, we denote $I_{\sigma}$ by $I_{r}$,
 $\omega_{\sigma}$ by $\omega_{r}$ and $\Delta_{\sigma}$ be $\Delta_{r}$.

 From Corollary \ref{mincan}, it follows that $|F|-c(F)$, $F \in \mathcal{NF}(\Delta_{\sigma})$ 
 are the total 
 degrees of the generators of the canonical ideal $\omega_{\sigma}$. We call the corners
 of the a non-flippable facet $F \in \mathcal{NF}(\Delta_{\sigma})$ the {\it non-flippable corners}. 
 In the case of the simplicial complex
$\Delta_{r}$, we will show
that the number $t$ of the non-flippable corners could be any integer between $r$ and $r(m-r)$.

\begin{Proposition}\label{cor} Let $\Delta_{r}$ be the simplicial complex with the
Stanley-Reisner ideal $I_{r}^{*}$. Then there exists a non-flippable facet $F$ of the
simplicial complex $\Delta_{r}$
with $t$ corners if and only if  $r \leq t \leq r(m-r).$ 
\end{Proposition}
\begin{proof} We will construct a non-flippable facet for any given number of corners between
$r$ and $r(m-r)$. As any facet $F$  of $\Delta_{\sigma}$ is a 
disjoint union of $r$ paths $C_i$ from $(i,n)$ to $(m,i)$, we notice that the minimum number 
of non-flippable corner for any path $C_i$ is one and the maximum is
$(m-r)$. Hence minimum and maximum number of possible total non-flippable corners are $r$ and 
$r(m-r)$ respectively.
As a path $C_i$ is determined by its corners, we define the non-flippable corners for each path.
For $r$ corners, we define $C_i$ such that $\mathcal{C}(C_i)=(i+1,i+1)$ such that 
$F=C_1\cup \cdots\cup C_r$ is a non-flippable facet with $r$ corners; see Figure \ref{pic1}.

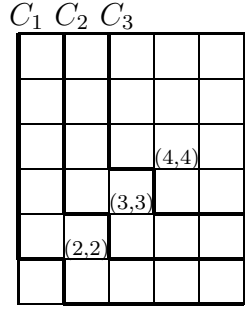
\begin{figure}[h]
\setlength{\unitlength}{6mm}
\begin{picture}(6,7)
\linethickness{0.15mm}
\multiput(0,0)(1,0){6}%
{\line(0,1){6}}
\linethickness{0.15mm}
\multiput(0,0)(0,1){7}%
{\line(1,0){5}}
\multiput(1,1.1)(1,0){1}%
{\tiny{(2,2)}}
\multiput(2,2.1)(1,0){1}%
{{\tiny(3,3)}}
\multiput(3,3.1)(1,0){1}%
{\tiny{(4,4)}}
\multiput(-0.2,6.2)(1,0){1}%
{$C_1$}
\multiput(0.8,6.2)(1,0){1}%
{$C_2$}
\multiput(1.8,6.2)(1,0){1}%
{$C_3$}
\linethickness{0.4mm}
\multiput(0,6)(0,1){1}%
{\line(0,-1){5}}
\linethickness{0.4mm}
\multiput(0,1)(1,0){1}%
{\line(1,0){1}}
\linethickness{0.4mm}
\multiput(1,1)(0,1){1}%
{\line(0,-1){1}}
\linethickness{0.4mm}
\multiput(1,0)(1,0){1}%
{\line(1,0){4}}
\linethickness{0.4mm}
\multiput(1,6)(0,1){1}%
{\line(0,-1){4}}
\linethickness{0.4mm}
\multiput(1,2)(1,0){1}%
{\line(1,0){1}}
\linethickness{0.4mm}
\multiput(2,2)(0,1){1}%
{\line(0,-1){1}}
\linethickness{0.4mm}
\multiput(2,1)(1,0){1}%
{\line(1,0){3}}
\linethickness{0.4mm}
\multiput(2,6)(1,0){1}%
{\line(0,-1){3}}
\linethickness{0.4mm}
\multiput(2,3)(1,0){1}%
{\line(1,0){1}}
\linethickness{0.4mm}
\multiput(3,3)(1,0){1}%
{\line(0,-1){1}}
\linethickness{0.4mm}
\multiput(3,2)(1,0){1}%
{\line(1,0){2}}
\end{picture}
\caption{\label{pic1}  A non-flippable facet with $r=3$ corners.}
\end{figure}

One can write any $r\leq t \leq r(m-r)$ as $t=r+p (m-r-1)+q$ for $0\leq p\leq r$ and $0\leq q < (m-r-1)$.
For any such $t$, we define the corners of the path $C_i$ as follows: 
For $0\leq k\leq p-1$, the path $C_{r-k}$ has corners at
$$\big(r-(k-1),n-(k+1)\big),\big(r-(k-2),n-(k+2)\big),\ldots,\big( r-(k-m+r),n-( k+m-r)\big).$$

The path $C_{r-p}$ has corners at $$(r-p,r-p+q),(r-p+1,r-p+q-1),\ldots,(r-p+q,r-p),$$ and 
for $1\leq i\leq r-p-1$, the path $C_i$ has corner at $(i+1,i+1)$.
Now $F=\bigcup_{i=1}^{r} C_i$ is a non-flippable facet with exactly $t=r+p(m-r-1)+q$
corners; see Figure \ref{pic2}. 
\begin{figure}[h]
\setlength{\unitlength}{6mm}
\begin{picture}(6,7)
\linethickness{0.15mm}
\multiput(0,0)(1,0){6}%
{\line(0,1){6}}
\linethickness{0.15mm}
\multiput(0,0)(0,1){7}%
{\line(1,0){5}}
\multiput(1,1.1)(1,0){1}%
{\tiny{(2,2)}}
\multiput(2,3.1)(1,0){1}%
{{\tiny(3,4)}}
\multiput(3,2.1)(1,0){1}%
{{\tiny(4,3)}}
\multiput(3,5.1)(1,0){1}%
{{\tiny(4,6)}}
\multiput(4,4.1)(1,0){1}%
{\tiny{(5,5)}}
\multiput(5,3.1)(1,0){1}%
{{\tiny(6,4)}}
\multiput(-0.2,6.2)(1,0){1}%
{$C_1$}
\multiput(0.8,6.2)(1,0){1}%
{$C_2$}
\multiput(1.8,6.2)(1,0){1}%
{$C_3$}
\linethickness{0.4mm}
\multiput(0,6)(0,1){1}%
{\line(0,-1){5}}
\linethickness{0.4mm}
\multiput(0,1)(1,0){1}%
{\line(1,0){1}}
\linethickness{0.4mm}
\multiput(1,1)(0,1){1}%
{\line(0,-1){1}}
\linethickness{0.4mm}
\multiput(1,0)(1,0){1}%
{\line(1,0){4}}
\linethickness{0.4mm}
\multiput(1,6)(0,1){1}%
{\line(0,-1){3}}
\linethickness{0.4mm}
\multiput(1,3)(1,0){1}%
{\line(1,0){1}}
\linethickness{0.4mm}
\multiput(2,3)(1,0){1}%
{\line(0,-1){1}}
\linethickness{0.4mm}
\multiput(2,2)(1,0){1}%
{\line(1,0){1}}
\linethickness{0.4mm}
\multiput(3,2)(1,0){1}%
{\line(0,-1){1}}
\linethickness{0.4mm}
\multiput(3,1)(1,0){1}%
{\line(1,0){2}}
\linethickness{0.4mm}
\multiput(2,6)(1,0){1}%
{\line(0,-1){1}}
\linethickness{0.4mm}
\multiput(2,5)(1,0){1}%
{\line(1,0){1}}
\linethickness{0.4mm}
\multiput(3,5)(1,0){1}%
{\line(0,-1){1}}
\linethickness{0.4mm}
\multiput(3,4)(1,0){1}%
{\line(1,0){1}}
\linethickness{0.4mm}
\multiput(4,4)(1,0){1}%
{\line(0,-1){1}}
\linethickness{0.4mm}
\multiput(4,3)(1,0){1}%
{\line(1,0){1}}
\linethickness{0.4mm}
\multiput(5,3)(1,0){1}%
{\line(0,-1){1}}
\end{picture}
\caption{\label{pic2}  A non-flippable facet with $t=r+p(m-r-1)+q$ corners with $m=6,n=7,r=3$ and $p=1,q=1$.}
\end{figure}
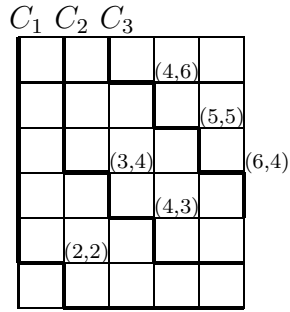

\end{proof}  \pagebreak[3]

\begin{Corollary}\label{total}  The 
canonical ideal $\omega_r$ 
has a minimal generator of degree $t$ if and only if  $rn\leq t\leq r(n+m-r-1)$. 
\end{Corollary}
\begin{proof} We have $\dim R/I_r=|F|=r(m+n)-r^2$, \cite[Corollary 5.18]{BV}. Now by Corollary \ref{mincan} and 
from Proposition \ref{cor}, follows the result.
\end{proof}

Next, we want to consider the boundary complex $\partial_r$ of  the simplicial complex $\Delta_r$. We want to 
show that the Stanley-Reisner ring $S/I_{\partial_r}$ satisfies the multiplicity conjecture. The geometric 
realization $|\partial_r|$ of the boundary complex $\partial_r$ is a sphere of dimension $r(m+n)-r^2-1$. 
Therefore the Stanley-Reisner ring  
$S/I_{\partial_{r}}$ is a Gorenstein ring, see \cite[Corollary 5.6.5]{BH}. Hence, the boundary complex
 $\partial_r$ satisfies properties (P1),\;(P2),\;(P3) of Section 1 and by Theorem \ref{hoch}, 
 we have $S/I_{\partial_r}=K[\Delta_r]/(\omega_r)$.
 
\begin{Theorem} The Stanley-Reisner ring $S/I_{\partial_r}$ satisfies the multiplicity conjecture.
\end{Theorem}
\begin{proof} We need to show that assumptions (A1) and (A2) are satisfied, see Theorem~\ref{mul}. 
As the  generators of the 
canonical ideal $\omega_r$ of $\Delta_r$ has degrees $t$ where $rn \leq t\leq r(m+n-r-1)$, 
there exists
a minimal inside face of dimension $r(m+n-r-1)-1=\dim R/I_{\partial_r}-(r+1)$  
and  there is no inside face of dimension less than $r+1$, see Theorem \ref{hoch}. Hence assumption
(A1) is satisfied. 

For Assumption (A2), we need to show that $h$-vector of $S/I_{\partial_r}$ is unimodal. 
Let the $h$-vector of the simplicial complex $\Delta_r$
 be given by $\big(h_0,\ldots,h_{r(m+n)-r^2}\big)$, then the $h$-vector 
 $\big(h'_0,\ldots,h'_{r(m+n)-r^2-1}\big)$ of the boundary complex $\partial_r$
 is given by (see \cite[Page 137]{ST}): $$h_0-h_{r(m+n)-r^2},\ldots,h_0+\cdots+
 h_{r(m+n)-r^2-1}-h_{r(m+n)-r^2}-\cdots-h_{1}.$$ 
By \cite[Theorem 2.4]{BAH} we have that $h_i$ calculates the number of facets $F$ 
of $\Delta_r$ with number of corners $c(F)=i$ and from Corollary \ref{cor}, we get that the 
maximal number of corners possible are $r(m-r)$, hence $h_{t}=0$ for all $r(m-r)+1\leq t\leq r(m+n)-r^2$.
Then it follows that the $h$-vector of $S/I_{\partial_r}$ is given by
$$
h'_i = 
\begin{cases}
h'_{r(m+n)-r^2-1-i}=\sum_{j=0}^{i}h_j & \text{for $i=0,\ldots,r(m-r)$;}\\\\
\sum_{j=0}^{r(m-r)}h_{j} & \text{for $j=r(m-r)+1,\ldots,nr-2$;}\\\\
\end{cases}
$$
Hence $h$-vector of $S/I_{\partial_r}$ is unimodal.
\end{proof} 
In the remaining part of this section, we compare the Stanley-Reisner ideal $I_{m-1}^*$
of $\Delta_{m-1}$ with its 
 $(I_{m-1}^{*})^{\vee}$. We will see in Theorem \ref{dual}
that the dual ideal $(I_{m-1}^{*})^{\vee}$ is again the initial ideal of the ideal of the maximal 
minors of a certain matrix.

Let $\Delta$ be a simplicial complex on the vertex set $[n]$
and $I_{\Delta}\subset K[X_1,\ldots,X_n]$ be the corresponding Stanley-Reisner ideal. 
There is another simplicial complex
$\Delta^{\vee}$ associated to $\Delta$ which is called the {\it Alexander dual} of $\Delta$.
The Alexander dual is defined by the simplicial complex
$\Delta^{\vee}=\{[n]\setminus F: F\notin \Delta\}$. It is easy to see that 
the complement of the minimal non-faces of the simplicial complex
 $\Delta$ define the facets of the dual complex $\Delta^{\vee}$ and vise-versa. Hence, the Stanley Reisner ideal
$I_{\Delta^{\vee}}$ is equal to the ideal 
$\big(X_{i_1}\cdots X_{i_k} : [n]\setminus\{i_1,\ldots,i_k\}\in \mathcal{F}(\Delta)\big)$. 
One may write $I_{\Delta}=\bigcap_{F \in \mathcal{F}(\Delta)}P_F$ where
$P_F=(X_i: i\notin F)$. Therefore the monomials $X_{P_F}=\prod_{X_i \in P_F} X_i$, $F\in \mathcal{F}(\Delta)$ 
form a set of minimal generators of $I_{\Delta^{\vee}}$. From here it follows that 
a monomial
$g$ is a minimal generator of $I_{\Delta^{\vee}}$ if and only if $\mathcal{S}=\{X_i:\;X_i|\;g\}$ 
is a vertex cover of the
 set of minimal generators $G(I_{\Delta})$ of $I_{\Delta}$ (We call a set of indeterminates
 $\mathcal{S}\subset\{X_1,\ldots,X_n\}$ to be vertex cover of a set of monomials $\{m_1,\ldots,m_k\}$
 if for all $m_i$ there exists some $X_j\in S$ such that $X_j|\;m_i$ ). 

 Let $X=(X_{ij})$ be a matrix of indeterminates of order $m\times n$. We call
 a matrix $Y=(Y_{ij})$ of indeterminates of order $(n-m+1)\times n$ a dual of the matrix $X$ if 
 $Y_{i,j+i-1}=X_{j,j+i-1}$ for $i=1,\ldots,n-m+1$ and $j=1,\ldots,m$. Notice that if $Y$
 is a dual of $X$, then $X$ is a dual of $Y$. For example, if \[X= \left( \begin{array}{cccc}
X_{11} & X_{12} & X_{13}& X_{14} \\
X_{21} & X_{22} & X_{23} & X_{24}\\
X_{31} & X_{32} & X_{33} & X_{34} \end{array} \right)\] is a matrix of order $3\times 4$
then a dual matrix $Y$ of order $2 \times 4$ can be defined as follows: 
 \[Y= \left( \begin{array}{cccc}
X_{11} & X_{22} & X_{33}& Y_{14} \\
Y_{21} & X_{12} & X_{23} & X_{34}\end{array} \right).\]

Let again $I_{m-1}^*$ denote the initial ideal of the ideal of maximal minors of an 
$m \times n$ matrix $X=(X_{ij})$ of indetermiantes
and $\Delta_{m-1}$ be the simplicial complex with Stanley-Reisner ideal $I_{m-1}^*$.
We denote the Alexander dual of the simplicial complex $\Delta_{m-1}$ by $\Delta_{m-1}^{\vee}$ and the 
corresponding Stanley-Reisner ideal by
$({I_{m-1}^{*}})^{\vee}$. Let $Y=(Y_{ij})$ be a dual matrix of $X$.
Let $J_{n-m}$ denote the ideal of the maximal minors of the matrix $Y$ and 
the initial ideal of $J_{n-m}$ be denoted by
$J_{n-m}^{*}$ (notice $J_{n-m}^*$ does not depend upon the
choice of the dual matrix $Y$). We define a polynomial ring
$T=K[X_{ij},Y_{kj}:\; 1\leq i\leq m, 1\leq k \leq n-m+1, 1\leq j\leq n ]$. Then we have:  
\begin{Theorem}\label{dual} $$({I_{m-1}^{*}})^{\vee}T=J_{n-m}^*T.$$
\end{Theorem}
\begin{proof} First we show
that the ideal $J_{n-m}^*T$ is contained in the ideal
$(I_{m-1}^{*})^{\vee}T$. 
Let $g=Y_{1j_1}Y_{2j_2}\cdots Y_{n-m+1,j_{n-m+1}}, \;j_1<j_2<\cdots<j_{n-m+1}$ be 
a minimal generator of the ideal $J_{n-m}^*$. As 
$Y_{1j}=X_{jj}, Y_{2j+1}=X_{jj+1}, \ldots,Y_{n-m+1,j+n-m}=X_{jj+n-m}$ for $j=1,\ldots,m$,
the monomial $g$ is of the form $X_{i_1,i_1}X_{i_2,i_2+1}\cdots X_{i_{n-m+1},i_{n-m+1}+n-m}$ for some
$ 1\leq i_1\leq i_2\leq\cdots\leq i_{n-m+1}\leq m.$ We need to show that the set
$S$ given by $\{X_{i_1,i_1},X_{i_2,i_2+1},\ldots, X_{i_{n-m+1},i_{n-m+1}+n-m}\}$ is a vertex 
cover for $G(I_{m-1}^{*})$. Let 
$$h=X_{1,1+t_1}X_{2,2+t_2}\cdots X_{m,m+t_m}, \;0\leq t_1\leq t_2\leq \cdots\leq t_m \leq n-m$$ 
be a minimal generator of $I_{m-1}^*$. 
We show that there exists $X_{i,j}\in S$ such that $X_{i,j}|\;h$. Suppose the contrary, then
$X_{i_k,i_{k}+(k-1)}$
does not divide $h$ for any $k=1,\ldots,n-m+1$ which implies $t_{{i_k}} >k-1$ for $k=1,\ldots,n-m+1$,
in particular $t_{i_{n-m+1}} >n-m$ which is a contradiction.

To show that $(I_{m-1}^{*})^{\vee}T\subset J_{n-m}^*T$, we need to show that if $S$ is a minimal 
vertex cover of $G(I_{m-1}^*)$, then $\prod_{X_{ij}\in S}X_{ij}$ is a generator of $J_{n-m}^*$. 
Since, the monomials $\prod_{i=1}^{m}{X_{i,i+k}},\; k=0,\ldots,n-m$ are minimal generators of
$G(I_{m-1}^*)$, we get that 
the subset of the form $S'=\{X_{i_1,i_1},X_{i_2,i_2+1},\ldots,X_{i_{n-m+1},i_{n-m+1}+n-m}\}$ 
is contained in any minimal vertex cover $S$ of $G(I_{m-1}^*)$. Also one may notice that, we must have
$1\leq i_1\leq i_2\leq\cdots\leq i_{n-m+1}\leq m$. Now, the generators of $J_{n-m}^*$
are exactly of the form $\prod_{X_{ij}\in S'}X_{ij}$, hence $(I_{m-1}^{*})^{\vee}T\subset J_{n-m}^*T$.
\end{proof}

\begin{Corollary} The Stanley Reisner Ideal $I_{m-1}^{*}$ has linear quotients.
\end{Corollary}
\begin{proof}  By above theorem and Theorem \ref{minor} we get
that the  simplicial complex $\Delta_{m-1}^{\vee}$ gives the triangulation of 
a shellable linear ball. Now it follows from Theorem 1.4 \cite{HHZ} that $I_{m-1}^*$ has linear quotients.
\end{proof}

\section{Polarization of the powers of a maximal ideal}

Let $S=K[x_1,\ldots,x_n]$ be a standard graded polynomial ring over the field $K$ and
let $\mm=(x_1,\ldots,x_n) \subset S$ denote the maximal graded ideal.

Let $u=\prod_{i=1}^{n} x_i^{a_i}$ be a monomial in $S$. Then the squarefree monomial given by
$$u^P=\prod_{i=1}^{n}\prod_{j=1}^{a_i}x_{ij}\in K[x_{11},\ldots,x_{1a_1},\ldots,x_{n1},\ldots,x_{na_n}]$$
is called the {\it polarization } of $u$.  Let $I=\mm^t$ be the $t$th power of the maximal ideal.
 Let $G(I)=\{u_1,\ldots,u_m\}$, then the
squarefree monomial ideal $I^P=(u_1^P,\ldots,u_m^P)
\subset K[x_{11},\ldots,x_{1t},\ldots,x_{n1},\ldots,x_{nt}]$ 
is called the {\it polarization} of $I$.
 
Let $\Gamma=\{a\in\NN^n\;:\; x^a\notin I \}$ be the multicomplex associated to the ideal $I$. The detailed 
information about multicomplexes can be found in \cite{HP}. In our case, $\Gamma$ is a shellable 
multicomplex, see \cite[Theorem 10.5]{HP} and all the elements of $\Gamma$
are its facets. Clearly, $\Gamma$ consists of those $a\in \NN^n$ such that $\sum a(k)\leq t-1$.
We define a partial order on the facets of $\Gamma$ as follows:
Let $a,b$ be any two facets of $\Gamma$,  we say $a <b$ if 
$\sum_{k=1}^{n}a(k) \leq \sum_{k=1}^{n}b(k)$. This partial order extended to any total 
order gives us a shelling. We fix
a total order and we call the respective shelling $\Sigma$. Let $\mathcal{F}(\Gamma)=
\{a_1,\ldots,a_m\}$ be the set of the facets of $\Gamma$
in the shelling order $\Sigma$. Let $\Delta$ be the simplicial complex with the
Stanley-Reisner ideal $I^P$ and let $\mathcal{F}(\Delta)$
be the set of facets of $\Delta$. By \cite{D}, it follows that $\Delta$ is shellable.
Furthermore by \cite[Lemma 3.7]{jahan} and \cite[Proposition 10.3]{HP}
together, it follows that there is a bijection between $\mathcal{F}(\Gamma)$
and $\mathcal{F}(\Delta)$ given by 
$$\theta: \mathcal{F}(\Gamma) \to \mathcal{F}(\Delta),\; a_k \mapsto F_{a_k}. $$
Here given the facet
$a_k=\big(a_k(1),\ldots,a_k(n)\big)$ of $\Gamma$, the facet  
$F_{a_k}$ of $\Delta$ is defined  to be $\{x_{ij},\; i=1,\ldots,n, j=1,\ldots,t, j\neq a_k(i)+1\}$. 
Also, $F_{a_1},\ldots,F_{a_m}$ is a shelling order of the facets of the simplicial complex $\Delta$.

We have the following:

\begin{Theorem}\label{polar}  The geometric realization $|\Delta|$   of
the simplicial complex  $\Delta$ is a shellable linear ball.
\end{Theorem}
\begin{proof} We already know that $\Delta=\langle F_{a_1},\ldots,F_{a_m}\rangle$ 
is a shellable simplicial complex. Note that the Stanley-Reisner ideal 
$I_{\Delta}=I^P$ has a linear resolution
because the graded Betti numbers of a monomial ideal and its polarization are the same, and $I=\mm^t$
obviously has a linear resolution.
Let $\Delta_k=\langle F_{a_1},\ldots,F_{a_k}\rangle$. 
We will prove $|\Delta_k|$ is a ball by induction on $k$ as in Theorem\ref{minor}.
The assertion is obvious for $k=1$. 
Assume that $|\Delta_{k-1}|$ is a ball, we will
show that $|\Delta_{k}|$ is a ball where the simplicial complex
$\Delta_{k}=\Delta_{k-1}\cup \langle F_{a_{k}} \rangle$. 
Let $\Delta_{k-1}\cap \langle F_{a_{k}}\rangle=\{G_1,\ldots,G_r\}$
where $G_1,\ldots,G_r$ are codimension one faces of $F_{a_{k}}$.
By Lemma \ref{topology}, we notice that $|\Delta_{k}|$ is a ball (assuming that $|\Delta_{k-1}|$ 
is a ball) if the following two conditions are satisfied:

\begin{enumerate}
\item Each $G_{\ell}$ is a subset of exactly one $F_{a_i}$ for $ i\leq k-1$, which in turn implies
that $G_{\ell} \in \partial{\Delta_{k-1}}$,
\item $G_1, \ldots, G_r$ is a proper subset of the boundary complex 
$\partial F_{a_{k}}$ of $F_{a_{k}}$.
\end{enumerate}
Let $a_{k}=(s_1,\ldots,s_n)$ where $\sum s_i\leq t-1$. Then 
$$F_{a_{k}}=\{x_{ij},\; i=1,\ldots,n, j=1,\ldots,t, j\neq s_i+1\}.$$ Suppose 
$G_{\ell}=F_{a_{k}}\setminus\{x_{i_{\ell}j_{\ell}}\}$ where $1\leq i_{\ell}\leq n$
and $1\leq j_{\ell}\leq t$.
Then clearly, $G_{\ell}=F_{a_{k}}\cap F_{a_{p_\ell}}$  where 
$a_{p_\ell}=(s_1,\ldots,s_{i_{\ell}-1},j_{\ell}-1,s_{i_{\ell}+1},\ldots,s_n)$ and also 
$G_{\ell} \not\subset F_{a_q}$ for any $q\leq k-1$,\;$q\neq {p_\ell}$.

For the second condition, let
$1 \leq q \leq n$ be the minimum integer such that $s_{q}< t-1$. 
Let $G=F_{a_{k}}\setminus \{x_{qt}\}$. Suppose $G\subset F_{a_j}$
 for some $j\leq k-1$, then it would imply that 
 $a_j=(s_1,\ldots,s_{q-1},t-1,s_{q+1},\ldots,s_n)$. Since 
 $\sum a_j(i)\geq t$, we have $a_j \notin \Gamma$, a contradiction.
 Hence $G \notin \{G_1,\ldots,G_r\}$ and $G$ is a facet of the boundary complex $\partial F_{a_k}$.

\end{proof}

Now by the above theorem and Corollary \ref{linmul}, we have the following:
\begin{Corollary} \label{polmul} The simplicial sphere $\partial\Delta$ 
satisfies the multiplicity conjecture.
\end{Corollary}

\end{document}